\documentclass[11pt]{amsart}
\usepackage{geometry}                % See geometry.pdf to learn the layout options. There are lots.
\usepackage{graphicx}
\usepackage{amssymb, amsmath, amsthm}
\usepackage{epstopdf}
\usepackage{enumerate}
\usepackage[vlined,ruled,linesnumbered]{algorithm2e}

\newtheorem{theorem}{Theorem}
\newtheorem{lemma}{Lemma}
\newtheorem{proposition}{Proposition}

\theoremstyle{definition}
\newtheorem{definition}{Definition}

\newcommand{\Z}{ {\mathbb Z}}
\newcommand{\Q}{{\mathbb Q}}

\newcommand{\lcm}{ {\rm lcm}}

\title{Fast tabulation of challenge pseudoprimes}
\author{Andrew Shallue and Jonathan Webster}
\thanks{The first author was supported in part by Illinois Wesleyan University's Artistic and Scholarly Development grant and the second author was supported in part by Butler University's Holcomb Awards Committee.}  

\begin{document}

\maketitle
\begin{abstract}
We provide a new algorithm for tabulating composite numbers which are pseudoprimes to both a Fermat test and 
a Lucas test.  Our algorithm is optimized for parameter choices that minimize the occurrence of pseudoprimes, 
and for pseudoprimes with a fixed number of prime factors.  Using this, we have confirmed that there 
are no PSW challenge pseudoprimes with two or three prime factors up to $2^{80}$.
In the case where one is tabulating challenge pseudoprimes with a fixed number of prime factors, 
we prove our algorithm gives an unconditional asymptotic improvement over previous methods.
\end{abstract}

% introduction
\section{Introduction}

Pomerance, Selfridge, and Wagstaff famously offered \$620 for a composite $n$ that satisfies
\begin{enumerate}
\item $2^{n-1} \equiv 1 \pmod{n}$  so $n$ is a base $2$ Fermat pseudoprime,
\item $( 5 \mid n ) = -1$  so $n$ is not a square modulo $5$, and 
\item $F_{n+1} \equiv 0 \pmod{n}$  so $n$ is a Fibonacci pseudoprime, 
\end{enumerate}
or to prove that no such $n$ exists.  We call composites that satisfy these conditions PSW challenge 
pseudoprimes.  
In \cite{PSW80} they credit R. Baillie with the discovery that combining a Fermat test with a Lucas test
(with a certain specific parameter choice) makes for an especially effective primality test \cite{BaillieWag80}.
Perhaps not as well known is Jon Grantham's offer of \$6.20 for a Frobenius pseudoprime $n$ 
to the polynomial $x^2 - 5x - 5$ with $( 5 \mid n) = -1$ \cite{grantham_challenge}.  Similar to the PSW challenge, Grantham's challenge number would be a base $5$ Fermat pseudoprime, a Lucas pseudoprime with polynomial $x^2 - 5x - 5$, and satisfy $(5 \mid n) = -1$.  Both challenges remain open as of this writing, 
though at least in the first case there is good reason to believe infinitely many exist \cite{PomHeuristic}.

The largest tabulation to date of pseudoprimes of similar type is that of Gilchrist \cite{Gilchrist13}, 
who found no Baillie-PSW pseudoprimes (a stronger version of the PSW challenge)
 up to $B = 2^{64}$.  After first tabulating
$2$-strong pseudoprimes \cite{Feitsma13, Nicely12} using an algorithm due to Pinch
\cite{Pinch00}, he applied the strong Lucas test using the code of Nicely \cite{Nicely12}.
Taking inspiration from tabulations of strong pseudoprimes to several bases
 \cite{jaeschke, bleichenbacher, 8bases,  SorWeb17}, our new idea is to treat the tabulation as a two-base computation: 
 a Fermat base and a Lucas base.  In this way we exploit both tests that make up the definition.

Specifically, we improve upon \cite{Pinch00} in three ways:
\begin{itemize}
\item GCD computations replace factorizations of $b^n -1$,
\item sieving searches are done with larger moduli, 
\item fewer pre-products are constructed.  
\end{itemize}

Other notable attempts to find a PSW challenge number involve construction techniques that result in a computationally infeasible 
subset-product problem \cite{Primes, psw_construct}.
The first of such attempts would have also found the number requested at the end of \cite{analogcarm}
which is simultaneously a Carmichael number and a $(P,Q)$-Lucas pseudoprime for all pairs $(P,Q)$ with $5 = P^2 - 4Q$ and $( 5 \mid n ) = -1$.  

The new algorithm presented constructs $n$ by pairing primes $p$ with admissible pre-products $k$.
In Section \ref{sec:runningtime} we provide an unconditional proof of the running time.  Unfortunately, 
the provable running time gets worse as the number of primes dividing $k$ increases.
Specifically, we prove the following.

\begin{theorem}
There exists an algorithm which tabulates all PSW challenge pseudoprimes up to $B$ with $t$ prime factors,
 while using $\widetilde{O}(B^{1 - \frac{1}{3t-1}})$ bit operations and space for $O(B^{\frac{3t-2}{4t-2}})$ words.  

The running time improves under a heuristic assumption that factoring plays a minimal role, 
to $\widetilde{O}(B^{1 - \frac{1}{2t-1}})$ bit operations.

No PSW challenge pseudoprimes with two or three prime factors exist up to $B = 2^{80}$.
\end{theorem}

For the computation performed we chose $2$ as the Fermat base and $(1,-1)$ as the Lucas base, 
but the algorithm as designed can handle arbitrary choices.

The rest of the paper is organized as follows. 
Section 2 establishes key definitions and notation, while
Section 3 provides the theoretical underpinnings of the algorithm.
The algorithm is presented in Section 4 along with a proof of correctness.
The running time is analyzed in Sections 5 and 6.
We conclude the paper with comments on our computation with $B = 2^{80}$. 

\section{Definitions and Notation}

A \textit{base $b$ Fermat pseudoprime} is a composite $n$ with $\gcd(n,b) = 1$ that satisfies the congruence $b^{n-1} \equiv 1 \pmod{n}$.  

Lucas sequences have many equivalent definitions.  We state a few important ones and let the reader consult standard sources 
such as \cite{Lehmer30}
for a more thorough treatment.  
Let $P,Q \in \mathbb{Z}$ and $\alpha, \beta$ be the distinct roots of $f(x) = x^2 - Px + Q$, with $D = P^2 - 4Q$
the discriminant.  Then the Lucas sequences are
\[ U_n(P,Q) = (\alpha^n - \beta^n)/(\alpha - \beta) \quad \mbox{and } V_n(P,Q) = \alpha^n + \beta^n \enspace . \]
Equivalently, we may define these as recurrence relations, where
\[ U_0(P,Q) = 0, \quad U_1(P,Q) = 1, \quad \mbox{and} \quad U_n(P,Q) = PU_{n-1}(P,Q) - QU_{n-2}(P,Q) \enspace .\]
and
$$
V_0(P,Q) = 2, \quad V_1(P,Q) = P, \quad \mbox{and} \quad V_n(P,Q) = P V_{n-1}(P,Q) - Q V_{n-2}(P,Q) \enspace .
$$
We will use $\epsilon(n) = (D \mid n)$ for the Jacobi symbol and will frequently write 
$U_n$ or $V_n$ when the particular sequence is clear from context.  It should be noted that the definition below guarantees that $n$ is odd so that the Jacobi symbol is well-defined.  Often $U_n$ is referred to as the Lucas sequence
with parameters $P$ and $Q$, but both $V_n$ and $U_n$ are needed for the ``double-and-add" method for computing 
$U_n$ using $O(\log{n})$ arithmetic operations.  For a more modern take on this classic algorithm
see \cite{JoyeQuis96}.

A \textit{$(P,Q)$-Lucas pseudoprime} is a composite $n$ with $\gcd(n,2QD) = 1$ such that $U_{n - \epsilon(n)} \equiv 0 \pmod{n}$.  

\begin{definition}
We call a composite $n$ a  $(b,P,Q)$-challenge pseudoprime if it is simultaneously a base $b$ Fermat pseudoprime, a $(P,Q)$-Lucas pseudoprime, and additionally satisfies $\epsilon(n) = -1$.  
\end{definition}

Note that $\epsilon(n) = -1$ means that $D$ is not a square.

A PSW challenge pseudoprime is then a $(2, 1, -1)$-challenge pseudoprime in our notation.  To get a Baillie-PSW pseudoprime, 
one replaces the Fermat test with a strong pseudoprime test and the Lucas test with a strong Lucas test.  The Lucas 
parameters are chosen as $P = 1$ and $Q = (1-D)/4$, where $D$ is the first discriminant in the sequence $\{5, -7, 9, -11, \dots\} = \{ (-1)^k(2k + 1) \}_{k \geq 2}$
for which $(D \mid n) = -1$.

We use $\ell_b(n)$ when $\gcd(b,n) = 1$ to denote the multiplicative order of $b$ modulo $n$, i.e. 
the smallest positive integer such that $b^{\ell_b(n)} = 1 \mod{n}$.  When $n = p$ is a prime, 
$\ell_b(p) \mid p-1$ by Lagrange's Theorem since $p-1$ is the order of $(\Z/p\Z)^{\times}$.

Given a prime $p$, there exists a least positive integer $\omega$ such that 
$U_{\omega} \equiv 0 \pmod{p}$.  We call $\omega$ the rank of apparition of $p$ with respect to 
the Lucas sequence $(P,Q)$, and we denote it by $\omega(p)$.  It is also well known that 
$U_{p - \epsilon(p)} \equiv  0 \pmod{p}$ and hence that $\omega(p) \mid p - \epsilon(p)$.

Throughout, we will use $\log$ to represent the natural logarithm.  

The function $P(n)$ returns the largest prime factor of $n$, and for asymptotic analysis we often use 
$\widetilde{O}$, where $f = \widetilde{O}(g)$
means there are positive constants $N, c$ such that $f(n) \leq g(n)(\log(4 + g(n)))^c$ 
for nonnegative functions $f(n)$ and $g(n)$ and for all $n \geq N$ \cite[Definition 25.8]{GathenGerhard03}.

% notation section
\section{Algorithmic Theory}

The main idea of the tabulation comes from \cite{jaeschke, bleichenbacher, 8bases,  SorWeb17}, 
but instead of tabulating pseudoprimes to many bases, we have just a Fermat base and a Lucas base.
For the Fermat case we state known results for completeness, while
for the Lucas case we state and prove the required results.
We follow the notation in \cite{ SorWeb17} when possible. 

To find all $(b,P,Q)$-challenge pseudoprimes $n < B$, we construct $n$ in factored form $n = p_1p_2 \ldots p_{t-1}p_t$ where $t$ is the number of prime divisors of $n$ and $p_i \leq p_{i+1}$. We call $k = p_1p_2\ldots p_{i}$ for $i < t$ a pre-product.  Subsection 3.1 states theorems limiting the number of pre-products that need to be considered.   Subsection 3.2 shows that $p_t$ may be found via a GCD computation when $k$ is small and by a sieving search when $k$ is large.

\subsection{Conditions on $n = wk$ }

We will frequently make use of the fact that if $\epsilon(n) = -1$ and $n = wk$ then $\epsilon(w) = -\epsilon(k)$ by the multiplicative property of the Jacobi symbol.  

%following taken from Bleichenbacher verbatim

\begin{proposition}[Theorem 3.20 of \cite{ bleichenbacher} ] \label{prop:weiferich}
Let $k \geq 1 $ be an integer and $p$ a prime.  
If $n = kp^2$ is a Fermat pseudoprime for the base $b$ then the following two conditions must be satisfied:
\begin{enumerate}
\item $b^{p-1} \equiv 1 \pmod{p^2}$,
\item  $b^{k-1} \equiv 1 \pmod{p^2}$.
\end{enumerate}
\end{proposition}

\begin{proposition} \label{prop:wallsunsun}
Let $k \geq 1 $ be an integer and $p$ a prime.  
If $n = kp^2$ is a $(P,Q)$-Lucas pseudoprime with $\epsilon(n) = -1$ then the following two conditions must be satisfied:
\begin{enumerate}
\item $U_{p-\epsilon(p)} \equiv 0 \pmod{p^2}$,
\item  $U_{k - \epsilon(k)} \equiv 0 \pmod{p^2}$.
\end{enumerate}
\end{proposition}
\begin{proof}
We start by noting that $\omega(p^2) \mid p \omega(p)$ and hence $\omega(p^2)$ 
divides $p(p - \epsilon(p))$ by the law of repetition \cite[Theorem 1.6]{Lehmer30}.  In addition, 
$U_{n+1} \equiv 0 \pmod{n}$ by assumption so that $U_{n+1} \equiv 0 \pmod{p^2}$
and hence $\omega(p^2) \mid n+1$.  With $p$ relatively prime to $n+1$, it follows that 
$\omega(p^2)$ divides $\gcd(n+1, p - \epsilon(p))$, and we conclude that $\omega(p^2)$ 
divides $p - \epsilon(p)$, which proves the first congruence.

For the second congruence, if $k=1$ then $U_{k - \epsilon(k)} = U_0$ and the congruence is satisfied.
In the case $k > 1$, $\omega(p^2)$ divides $n+1 = kp^2 +1 = kp^2 - \epsilon(k)$
and $p - \epsilon(p)$.  Thus $\omega(p^2)$ divides
$$
kp^2 - \epsilon(k) - k(p-1)(p+1) = kp^2 - \epsilon(k) - k(p^2-1) = k - \epsilon(k) \enspace .
$$
It follows that $U_{k - \epsilon(k)} \equiv 0 \pmod{p^2}$.
\end{proof}

In the case $b = 2$, these primes are known as Weiferich primes and in the $(1,-1)$ case they are known as Wall-Sun-Sun primes.
\cite{pom_weif} suggests the following heuristic argument to understand the rarity of these primes.  
Consider either $b^{p-1} -1$ or $U_{p-\epsilon(p)}$ in a base $p$ representation.  
The constant coefficient is zero by Fermat's Little Theorem and its analogue.
The coefficient on $p$ needs to be $0$ to satisfy the above congruence and we expect this to happen with probability $1/p$.
Summing over the reciprocal of primes gives an expected count of such primes up to $x$ as being on the order of $\log \log x$.
For challenge pseudoprimes, both congruences would have to be met simultaneously.
The corresponding count from the expected values is now a sum of $1/p^2$ and the infinite sum converges. 
So we expect the count to be finite and we know of no examples of this behavior.

Either the Fermat case or the Lucas case can individually be checked up to a bound $B$ in $O(B^{1/2})$ time and such primes may be then tested against the other condition.  In the very unlikely scenario that such a prime does exist, we refer the reader to section 6 of \cite{Pinch00} in order to account for square factors dividing challenge pseudoprimes.  Given how exceedingly rare we believe these are, we deal no further with square factors and assume a squarefree challenge pseudoprime.  

\begin{proposition} \label{prop:admissible}
Let $n = p_1p_2 \ldots p_t$ be a $(b,P,Q)$-challenge pseudoprime, 
\[ L = \lcm( \ell_b(p_1), \ldots, \ell_b(p_t)), \quad \mbox{and} \quad  W = \lcm( \omega(p_1), \ldots, \omega(p_t) ) 
\enspace . \]  
Then $\gcd(L, W) \leq 2$, $\gcd(n, L) =1$, and  $\gcd(n, W) = 1$ .
\end{proposition}
\begin{proof}
We have $b^{n-1} \equiv 1 \pmod{p_i}$ and hence $n \equiv 1 \pmod{\ell_b(p_i)}$.
We also have $U_{n+1} \equiv 0 \pmod{p_i}$ and hence $n \equiv -1 \pmod{\omega(p_i)}$.  
So  $\ell_b(p_i) \mid (n-1) $ and $\omega(p_i) \mid (n+1)$ and this holds for all $p_i \mid n$.
Therefore, $L\mid (n-1)$ and $W \mid (n+1)$. 
Then $\gcd(L, W) \leq  \gcd( n-1, n+1) \leq 2$.  
Since $n$ is relatively prime to both $n+1$ and $n-1$, the other two gcds are as claimed.
\end{proof}

This is extremely useful in limiting the pre-products under consideration.  For one, it means that most primes with $\epsilon(p) = 1$ need not be considered, since it is highly probable that $\gcd(\ell_b(p), \omega(p) ) > 2 $ when $\epsilon(p) = 1$.   In private correspondence,  Paul Pollack gave a heuristic argument suggesting around $\log(x)$ such primes up to $x$.    We call $k$ \textit{admissible} if the primes dividing $k$ satisfy the above proposition.

\subsection{Conditions on $p_t$ given $k$}

Henceforth, we assume that $k = p_1\ldots p_{t-1}$ and that $k$ is admissible.  

\begin{proposition}\label{prop:gcd}
If $n = kp$ is a $(b,P,Q)$-challenge pseudoprime then $p$ is a divisor of  
\[ \gcd(b^{k-1}-1, U_{k -\epsilon(k) }).\]
\end{proposition}

\begin{proof}
Recall that 
$b^{n-1} \equiv 1 \pmod{n}$ and $U_{n+1} \equiv 0 \pmod{n}$.
We rewrite $n-1 = kp - 1 = k(p-1) + k-1$.  
Since $\ell_b(p)$ divides  $(p-1)$ and $n-1$ we conclude $\ell_b(p) \mid k-1$.  
Thus, $p \mid b^{k-1} - 1$.

Similarly $n+1 = kp - \epsilon(p)\epsilon(k) = k(p - \epsilon(p)) + k\epsilon(p) -  \epsilon(p)\epsilon(k) =  k(p - \epsilon(p)) + \epsilon(p)( k - \epsilon(k))$.
Since $\omega(p)$ divides $p - \epsilon(p)$ and $n+1$ we conclude $\omega(p) \mid (k - \epsilon(k))$.
Thus, $p \mid U_{k-\epsilon(k)}$.
\end{proof}

\begin{proposition} \label{prop:sieve}
If $n = kp$ is a $(b,P,Q)$-challenge pseudoprime then 
\[ p \equiv \left\{ \begin{array}{c} k^{-1} \pmod{L} \\ -k^{-1} \pmod{W} \end{array} \right. ,\]
where
\[ L = \lcm( \ell_b(p_1), \ldots, \ell_b(p_{t-1})), \quad \mbox{and} \quad  W = \lcm( \omega(p_1), \ldots, \omega(p_{t-1}) ) 
\enspace . \]  

\end{proposition}
\begin{proof}
Since $n = kp$ is a challenge pseudoprime, we have that $b^{kp-1} \equiv 1 \mod{p_i}$
where $p_i$ is any prime factor of $k$, and so $\ell_b(p_i) \mid kp-1$.  
Thus, $p \equiv k^{-1} \mod{\ell_b(p_i)}$.
We also know that $\omega(n+1) \equiv 0 \mod{n}$, and hence that it is congruent to $0$ modulo $p_i$.
Thus, $\omega(p_i) \mid kp+1$ so that $p \equiv -k^{-1} \mod{\omega(p_i)}$.

Now, $\ell_b(p_i) \mid kp-1$ for all $p_i \mid k$ if and only if $L \mid kp-1$.
A similar statement holds for $W,$ which completes the proof.
\end{proof}

%%%%%%%%%%%%
%% new section on the algorithm
%%%%%%%%%%%%%%
\section{Algorithm}

Our basic strategy follows that found in \cite{SorWeb17}.
Find all pseudoprimes with $t$ prime factors for each $t \geq 2$ in turn.
For a given $t$, we analyze all pre-products $k$ with $t-1$ prime factors.  The question for 
each pre-product is whether there exists a prime $p$ such $n = kp$ is a challenge pseudoprime.
For small pre-products, this question can be answered with a $\gcd$ computation.
For large pre-products, we instead use a sieve.

% main algorithm
\begin{algorithm} \label{main_algorithm}
\caption{Tabulating squarefree challenge pseudoprimes}
\SetKwInOut{Input}{Input} \SetKwInOut{Output}{Output}

\Input{bound $B$, positive integer $b \geq 2$, Lucas sequence parameters $(P,Q)$}
\Output{list of $n \leq B$ which are $(b, P,Q)$-challenge pseudoprimes}
\BlankLine

Create an array of size $\sqrt{B}$ with entry $i$ containing the smallest prime factor of $i$\;
\For{primes $p \leq \sqrt{B}$}{
  Compute $\ell_b(p)$, $\omega(p)$ and only keep prime $p$ if $\gcd(\ell_b(p), \omega(p)) \leq 2$\;
  Update pre-product list\;
  \For{new pre-products $k$}{
  \If{$k \leq X$}{do {\bf GCD} step}
    \Else{do {\bf Sieve} step}
  } % endFor
} %endFor
\end{algorithm}

The above suggests storing all such primes up to $\sqrt{B}$ along with allowable pre-products, but  space constraints would prohibit this strategy in practice.  Construction of composite pre-products may be done with a combination of storing the 3-tuple $(p, \ell_b(p), \omega(p))$ for small primes and creating them on the fly for large primes, where the distinction is dependent upon  space constraints. To efficiently create them, one may use an incremental sieve or a segmented sieve to generate factorizations of consecutive integers so that we may quickly compute $\ell_b(p)$ from the factorization of $p-1$ and $\omega(p)$ from the factorization of $p-\epsilon(p)$.  

To tabulate Baillie-PSW pseudoprimes, one tabulates all pseudoprimes for each $D$ in the sequence.  Each discriminant 
performs a trial division so that successive computations will remove the next small prime from consideration, making the 
algorithm progressively more efficient.

%The algorithm can be modified for tabulating specifically Baillie-PSW which are strong challenge pseudoprimes.  Use Proposition 5 for the %gcd and Theorems 2 and 3 and Proposition 6 for the sieve.  Once one $(P,Q)$ pair is complete, move to the next pair.   Each discriminant %plays the role of a trial division so that successive computations will remove the next small prime from consideration.   Sieving also gets %progressively more efficient.  You construct $\epsilon(n) = 1$ with respect to the previous discriminants but $\epsilon(n) = -1$ with respect %to the current discriminant.  

\subsection{Algorithm Details and Correctness Proof}

We update the pre-product list as follows.  For each existing admissible pre-product $k'$, create a 
new pre-product $k = k'p$ and check that it is also admissible.  Recall that $k = \prod p_i$
is admissible if $\gcd(L, W) \leq 2$ where $L = \lcm_i( \ell_b(p_i))$ and $W = \lcm_i (\omega(p_i))$.

The GCD step involves computing and then factoring $\gcd(b^{k-1}-1, U_{k - \epsilon(k)})$.
For each prime $p$ dividing the $\gcd$ with $p > P(k)$, we build $n = kp$ and apply the Fermat test and the Lucas test
to determine if it is a challenge pseudoprime.  Importantly, both $b^{k-1}$ and $U_{k - \epsilon(k)}$ 
can be computed using a standard ``double-and-add" strategy at a cost of $O(\log{k})$ arithmetic operations.
With such large inputs, it is vital to use a gcd algorithm asymptotically faster than the Euclidean algorithm.
The solution is a discrete fast Fourier transform method that requires $\widetilde{O}(n)$
operations on $n$-bit inputs \cite{SteZim04}.

For the sieve step, we check primes $p$ in the range $p_{t-1} < p < B/k$ that fall into the arithmetic progression 
given by Proposition \ref{prop:sieve}.  For each such prime, we again construct $n = kp$ 
and apply the tests directly to see if it is a challenge pseudoprime.

% correctness theorem
\begin{theorem}
Algorithm \ref{main_algorithm} correctly tabulates all squarefree 
$(b,P,Q)$-challenge pseudoprimes up to $B$.
\end{theorem}
\begin{proof}
Suppose that $n \leq B$ is a $(b,P,Q)$-challenge pseudoprime.  Then we can write 
$n = p_1 \cdots p_t = k p_t$.  By Proposition \ref{prop:admissible}, $\gcd(L,W) \leq 2$, 
and this is true whether $L,W$ are computed for each of the $p_i$ separately, 
for $k$, or for $n$ as a whole.  Thus, limiting our pre-product list to admissible $k$ is valid.
Note that any prime $p \mid k$ satisfies $p \leq B^{1/2}$, so finding all primes up to $B^{1/2}$ 
is sufficient, if space intensive.

Given $k$, it follows from Propositions \ref{prop:gcd} and \ref{prop:sieve}
that $p_t$ is a divisor of~$\gcd(b^{k-1}-1, U_{k - \epsilon(k)})$ and that 
$$
p_t \equiv \left\{ \begin{array}{c} k^{-1} \pmod{L} \\ -k^{-1} \pmod{W} \end{array} \right. \enspace .
$$
Note that $k^{-1}$ exists modulo $L$ and modulo $W$ because $\gcd(n, L) = \gcd(n, W) = 1$.
Thus, the algorithm will find $p_t$ either through the GCD step or the Sieve step.

Finally, there is no chance of false positives because each potential pseudoprime is subjected 
to the necessary Fermat and Lucas tests.\end{proof}

%%%%%%%%%%%%%%%%
%%% new section
%%%%%%%%%%%%

% section with analytic number theory
\section{Reciprocal sums involving order} \label{sec:reciprocalsums}

The next two sections develop a proof of the asymptotic running time in the case where $t = 2$ or $t = 3$.
This proof depends on finding upper bounds on the sum over primes
$$
\sum_p \frac{1}{p \cdot \lcm(\ell_b(p), \omega(p))} \enspace .
$$
Since such results are of independent interest, we spend some time here developing 
the appropriate theory.  A general observation is that in order to bound a reciprocal sum of a function $f(n)$, 
it is not sufficient to know that $f(n)$ is usually large.  Instead, we need a precise bound on how often 
$f(n) \leq y$ for a range of values $y$.

The first step is to prove a slight generalization of a known lemma.  Our proof will follow closely 
the version found as Lemma 3 in \cite{Murty88}.  Let $b$ be the base of the Fermat test, 
and let $\beta = \alpha/\bar{\alpha}$ where $\alpha, \bar{\alpha}$ are the roots of 
$x^2 - Px+Q$.  In this context let $D$ be the squarefree part of the discriminant of $x^2 - Px + Q$.
Define $\Gamma$ as the subgroup of the unit group of $\Q(\sqrt{D})$
generated by $\beta$, and let $\Gamma_p$ be the reduction of $\Gamma$ modulo $p$.

% lemma on counting the number of primes with small  of apparition
\begin{lemma}\label{lem:boundedrank}
Let $\Gamma$ be a rank $1$ subgroup of $\Q(\sqrt{D})$, generated by $\beta$.
Then there are $O(y^2)$ primes $p$ such that $|\Gamma_p| \leq y$.
\end{lemma}
\begin{proof}
Let $n$ be a positive integer less than $y$, and consider $\beta^n- 1$.  Since $\beta \in \Q(\sqrt{D})$, 
so is $\beta^n - 1$.  Analyzing the numerator, it is straightforward to show that the numerator of 
$\beta^n - 1$ is at most $c^n$, where $c$ is a constant depending on $P$ and $Q$.

Now, define $S = \{\beta^n \ : \ 0 \leq n \leq y\}$.  If $|\Gamma_p| \leq y$ then two elements of $S$ 
are equal modulo $p$, i.e. $\beta^{n_1} = \beta^{n_2} \mod{p}$.  Without loss of generality, assume $n_1 \geq n_2$
so that $m = n_1 - n_2$ is nonnegative.  Then $\beta^{n_1 - n_2} = 1 \mod{p}$ and we denote $m = n_1-n_2$, 
noting that $0 \leq m \leq y$.   Then thinking of $\beta^m - 1$
as an element of $\Q(\sqrt{D})$, we have $\beta^m - 1 = \gamma_1 + \gamma_2 \sqrt{D}$, and 
$\beta^{n_1 - n_2} = 1 \mod{p}$ implies $p$ divides the numerators of the rational numbers $\gamma_1$
and $\gamma_2$.

For any given $m = n_1 - n_2 \leq y$, there are $O(m) = O(y)$ primes dividing the numerators of both $\gamma_1$ 
and $\gamma_2$, where the constant depends on the choice of $\beta$.  Thus, the total number of primes 
with $|\Gamma_p| \leq y$ is $O(y^2)$.
\end{proof}

The next lemma will be essential in the analysis of the sieve step of Algorithm \ref{main_algorithm}.
The authors are very grateful to an anonymous referee for suggesting the usage of 
the Cauchy-Schwarz inequality, thus improving the bound from 
$\widetilde{O}( X^{-2/3})$ to $\widetilde{O}(X^{-1})$.

% proposition on the reciprocal sum of p ell_a(p) omega(p)
\begin{lemma}\label{lem:sieve_reciprocal}
We have 
$$
\sum_{X < p < B \atop \gcd(\ell_b(p), \omega(p)) \leq 2} \frac{1}{p \cdot \lcm(\ell_b(p), \omega(p))}
= \widetilde{O}(X^{-1})
$$
where the sum is over primes and the implicit logarithm factor depends on $B, b, P, Q$.
\end{lemma}
\begin{proof}
We first utilize the fact that $\gcd(\ell_b(p), \omega(p)) \leq 2$ for all primes in the sum, 
along with the Cauchy-Schwarz inequality to get the new upper bound
$$
\sum_{X < p < B} \frac{2}{p \cdot \ell_b(p) \omega(p)}
\leq \left( \sum_{X < p < B} \frac{1}{p \cdot \ell_b(p)^2} \right)^{1/2}
\left( \sum_{X < p < B} \frac{1}{p \cdot \omega(p)^2} \right)^{1/2} \enspace .
$$
To bound these new sums, we break into two pieces depending on whether $\ell_b(p)$ is greater 
or less than $y$ (similarly, whether $\omega(p)$ is greater or less than $y$).

In the case where $\ell_b(p)$ is small we will use partial summation, and thus require a bound 
on the count of primes $p$ with $\ell_b(p) \leq y$.  By Murty-Srinivasan, Lemma 1, we know 
there are $O(y^2)$ primes with $\ell_b(p) \leq y$.  Using partial summation, we then have 
$$
\sum_{X < p < B \atop \ell_b(p) \leq y} \frac{1}{\ell_b(p)^2}
= \frac{1}{y^2} \cdot O(y^2) - \int_1^y O(t^2) \cdot -2 t^{-3} \ {\rm d}t
= O(1) + O(\log{y}) 
$$
and so
$$
\sum_{X < p < B \atop \ell_b(p) \leq y} \frac{1}{p \cdot \ell_b(p)^2}
\leq \frac{1}{X} \sum_{X < p < B \atop \ell_b(p) \leq y} \frac{1}{\ell_b(p)^2}
\leq O\left( \frac{\log{y}}{X} \right) \enspace .
$$

In the case where $\ell_b(p)$ is large we bound as follows:
$$
\sum_{X < p < B \atop \ell_b(p) > y} \frac{1}{p \cdot \ell_b(p)^2}
\leq \frac{1}{y^2} \sum_{X < p < B} \frac{1}{p} \leq O\left( \frac{\log{B}}{y^2} \right) \enspace .
$$

Balancing the two cases gives $\sum_{X < p < B} 1/(p \ell_b(p)^2) = \widetilde{O}(X^{-1})$.

By Lemma \ref{lem:boundedrank}, there are also at most $O(y^2)$ primes with $\omega(p) \leq y$.
Using the same argument as above, we also have $\sum_{X < p < B} 1/(p \omega(p)^2) = \widetilde{O}(X^{-1})$.
The result then follows.
\end{proof}

%%%%%%%%%%%
% new section
%%%%%%%%%%
\section{Algorithm Analysis} \label{sec:runningtime}

In this section we provide an asymptotic analysis of Algorithm 1.
Recall the additional assumption that the squarefree part of $D$ is not $-1$ or $-3$.
First we find the cost of the GCD step.

% theorem with gcd analysis
\begin{theorem} \label{thm:gcdcost}
The asymptotic cost of the $\gcd$ step for all $k \leq X$ is 
$\widetilde{O}(X^2) + \widetilde{O}(B^{1/2} X^{3/2})$ bit operations and space for 
$\widetilde{O}(B^{1/2} X^{1/2})$ words.
\end{theorem}
\begin{proof}
As noted above, for each pre-product $k \leq X$ we need to compute $b^{k-1}-1$
and $U_{k-\epsilon(k)}$ at a cost of $\widetilde{O}(k)$ bit operations, then apply a linear $\gcd$ algorithm
to compute $g(k) = \gcd(b^{k-1}-1, U_{k - \epsilon(k)})$ at a cost of $\widetilde{O}(k)$ bit operations.

In factoring $g(k)$ we do not need a complete factorization; rather we need to find all primes $p < B/k$ 
that divide $g(k)$.  Using the polynomial evaluation method of Pollard and Strassen
(see \cite[Theorem 19.3]{GathenGerhard03}) this requires $\widetilde{O}((B/k)^{1/2} \cdot \log(g(k)))
= \widetilde{O}( (B k)^{1/2})$ bit operations and $O( (Bk)^{1/2})$ space.

The total cost in bit operations for all $k \leq X$ is then
$$
\sum_{k \leq X} O(k) + \widetilde{O}(k) + \widetilde{O}( (Bk)^{1/2})
 = \widetilde{O}(X^2) + \widetilde{O}( B^{1/2} X^{3/2}) \enspace .
$$
\end{proof}

Next we find the cost of the Sieve step of Algorithm 1, broken down 
by the number of prime factors in the pre-product.

% new theorem
\begin{theorem} \label{thm:sievecost}
Restrict attention to the tabulation of $(b, P, Q)$-challenge pseudoprimes that are squarefree 
with $t \geq 3$ prime factors.  Then the cost in bit operations of the Sieve step in Algorithm 1
is 
$$
\widetilde{O}(X^{-1/(t-1)} B) \enspace .
$$

\end{theorem}
\begin{proof}
By construction we have $n = k p_t$ where $k > X$ and $p_t$ is the largest prime factor dividing $n$.
Since $k$ is admissible, $\gcd(\ell_b(p), \omega(p)) \leq 2$ for all $p \mid k$.

Let $k'$ denote $k/p_{t-1}$, the product of the smallest $t-2$ primes in the pre-product.
It follows that $X < k < B^{1-1/t}$ and so $\frac{X}{k'} < p_{t-1} < \frac{B^{1-1/t}}{k'}$.
As $t$ increases, $k'$ might become larger than $X$.  In this case we use the alternate lower 
bound $p_{t-1} > X^{1/(t-1)}$.  This lower bound is true because we construct $k$ so that its 
prime factors are increasing, and thus if $p_{t-1} \leq X^{1/(t-1)}$ then $k \leq X$, a contradiction.

By Proposition \ref{prop:sieve} the size of the arithmetic progression to check for each pre-product $k$ 
is $\frac{B}{k \lcm(L,W)}$, where $L$ and $W$ are computed from the primes dividing $k$.
Then the total cost in arithmetic operations for all pre-products with $t-1$ prime factors is 
\begin{align*}
\sum_{X < k < B^{1-1/t}} \frac{B}{k \lcm(L,W)}
&\leq \sum_{k' \leq X^{1 - \frac{1}{t-1}}} \ \sum_{\frac{X}{k'} < p_{t-1} < \frac{B^{1-1/t}}{k'}} \frac{B}{k' p_{t-1} \lcm(\ell_b(p_{t-1}), \omega(p_{t-1}))} \\
&+ \sum_{X^{1 - \frac{1}{t-1}} < k' < B^{1 - \frac{2}{t}}} \ \sum_{X^{\frac{1}{t-1}} < p_{t-1}} \frac{B}{k' p_{t-1} \lcm(\ell_b(p_{t-1}), \omega(p_{t-1}))} \enspace .
\end{align*}
For both sums the key tool will be Lemma \ref{lem:sieve_reciprocal}.  In the first case we have
\begin{align*}
\sum_{k' \leq X^{1 - \frac{1}{t-1}}} \ \sum_{\frac{X}{k'} < p_{t-1} < \frac{B^{1-1/t}}{k'}} \frac{B}{k' p_{t-1} \lcm(\ell_b(p_{t-1}), \omega(p_{t-1}))} 
& \leq \sum_{k' < X^{1 - \frac{1}{t-1}}} \frac{B}{k'} \cdot \widetilde{O}\left( \frac{k'}{X} \right) \\
&= \widetilde{O} \left( \frac{B}{X^{\frac{1}{t-1}}} \right)
\end{align*}
while in the second case we have 
\begin{align*}
 \sum_{X^{1 - \frac{1}{t-1}} < k' < B^{1 - \frac{2}{t}}} \ \sum_{X^{\frac{1}{t-1}} < p_{t-1}} \frac{B}{k' p_{t-1} \lcm(\ell_b(p_{t-1}), \omega(p_{t-1}))} 
& \leq \sum_{X^{1 - \frac{1}{t-1}} < k' < B^{1 - \frac{2}{t}}} \frac{B}{k'} \cdot \widetilde{O}(X^{-\frac{1}{t-1}}) \\
& = \widetilde{O} \left( \frac{B}{X^{\frac{1}{t-1}}} \right) \enspace .
\end{align*}
Since these arithmetic operations are on integers of size at most $B$, the result follows.
\end{proof}

Note that we are only utilizing the order statements for one prime in the pre-product; utilizing 
more seems quite difficult.

If the pre-product is prime and the pseudoprimes have two prime factors then the sum is easier to analyze, namely
$$
\sum_{X < q < B \atop \gcd(\ell_b(q), \omega(q)) \leq 2} \frac{B}{q \lcm(\ell_b(q), \omega(q))}
$$
which is $\widetilde{O}(B/X)$ by Lemma \ref{lem:sieve_reciprocal}.

These two theorems form the main components of the analysis of Algorithm \ref{main_algorithm}.

% algorithm running time
\begin{theorem}
The worst-case asymptotic running time of Algorithm 1, when restricted to constructing pseudoprimes 
with $t$ prime factors, is $\widetilde{O}(B^{1 - \frac{1}{3t-1}})$
bit operations.

The running time improves under a heuristic assumption that computing the gcd in the GCD step 
is more costly than factoring the gcd.  The running time becomes $\widetilde{O}(B^{1 - \frac{1}{2t-1}})$
bit operations when constructing $(b,P,Q)$-challenge pseudoprimes with $t$ prime factors.
\end{theorem}
\begin{proof}
We balance the cost of the GCD step from Theorem \ref{thm:gcdcost} and the cost of the 
Sieve step from Theorem \ref{thm:sievecost}.  
The bottleneck in the GCD step is factoring, and balancing $B/X$ with $B^{1/2} X^{3/2}$
gives $X = B^{1/5}$ and a running time with main term $B^{4/5}$ in the case $t = 2$.
In practice, computing gcds was the bottleneck rather than factoring.  If we assume this holds in 
general, the cost of the GCD step is instead $\widetilde{O}(X^2)$.
In the case $t=2$,
balancing $X^2$ with $B/X$ gives $X = B^{1/3}$ and a running time with main term 
$B^{2/3}$.  

For larger $t$, balancing $B X^{-\frac{1}{t-1}}$ with $B^{1/2} X^{3/2}$ 
gives $X = B^{\frac{t-1}{3t-1}}$ and a running time of $\widetilde{O}(B^{1 - \frac{1}{3t-1}})$ 
bit operations.  Under the heuristic assumption that the cost of the GCD step is instead $O(X^2)$, 
balancing with $B X^{-\frac{1}{t-1}}$ instead gives $X = B^{\frac{t-1}{2t-1}}$ and a running time 
of $\widetilde{O}(B^{1 - \frac{1}{2t-1}})$.

Asymptotically smaller is the cost of finding all primes up to $B^{1/2}$.  Applying 
the Fermat test and Lucas test to each composite constructed requires only
$O(\log{B})$ arithmetic operations per number on integers with $O(\log{B})$ bits.
\end{proof}

% new section
\section{Computational Notes and Conclusion}

We implemented Algorithm 1 and  
verified there are no $(2, 1, -1)$-challenge pseudoprimes (i.e. PSW challenge pseudoprimes)
with two or three prime factors less than $2^{80}$.  Since there are no primes up to $2^{40}$
which are simultaneously Weiferich and Wall-Sun-Sun, this claim includes composites with square factors.

If such a challenge pseudoprime with two prime factors were to be found, one of the primes 
would be admissible while satisfying $\epsilon(p) = 1$.  This would be a surprising occurrence for 
the following reason.  If $\epsilon(p) = 1$ then $\ell_b(p) \mid p-1$ and $\omega(p) \mid p-1$.
Since $\ell_b(p)$ and $\omega(p)$ are usually large, it will usually happen that $\gcd(\ell_b(p), \omega(p)) > 2$.
Thus it is notable that we found $7$ admissible primes with $\epsilon(p) = 1$ 
while generating primes less than $2^{40}$.

\begin{center}
\begin{tabular}{|r|r|r|}
\hline
$p$ & $\ell_2(p)$  & $\omega(p)$ \\ \hline
61681 & 40 & 1542\\
363101449& 171436 & 1059\\
4278255361 &   80 &  6684774\\
4562284561& 120& 147934\\
4582537681& 160453& 1428\\
26509131221& 748& 14176006\\
422013019339&  290442546 & 2906\\ \hline
\end{tabular}
\end{center}

One of the reasons the $(b,P,Q)$ test is effective is because of conflicting divisibility conditions.  
The Fermat condition requires divisibility with respect to  $n-1$.
The Lucas condition (with $\epsilon(n) = -1$) requires divisibility with respect to $n+1$.
Seemingly, this conflict will happen independent of the bases chosen.
However, $2047$ can be checked to be a $(2, 23, 131)$-challenge pseudoprime.
The authors are curious how challenging such pseudoprimes are in general.
Are there  bases for which the subset-product method of construction makes the challenge 
only moderately challenging?

The authors also note the influence on this problem of the number sought at the end of \cite{analogcarm}.  
That number is simultaneously a Carmichael number, a Lucas pseudoprime to all sequences of a fixed discriminant, and has $\epsilon(n) = -1$, 
so it would certainly be a challenge pseudoprime.  
Williams shows that such a number has
 an odd number of prime factors, has more than three prime factors, and is not divisible by 3.
 
We conclude by offering our own rewards for exhibiting challenge pseudoprimes:
\begin{itemize}
\item \$20 for a $(2,1,-1)$ challenge pseudoprime with an even number of prime factors,
\item  \$20 for a $(2,1,-1)$ challenge pseudoprime with exactly three prime factors,
\item \$6 for a $(2,1,-1)$ challenge pseudoprime divisible by $3$.
\end{itemize}

%%%%%%
% Idea from Webster: move the parameters just enough to make the Erdos construction work
%%%%%

%%%%%%%%%%
% To do
%
% Weiferich comment
%
%%%%%%%%%%%%%%%
\bibliography{bailliepsw}
\bibliographystyle{amsalpha}

\end{document}